\documentclass[a4paper,10pt]{article}
\usepackage{amssymb,amsmath,amsthm}
\usepackage{fullpage}
\usepackage{hyperref}
\usepackage{eucal}
\usepackage{color}
\usepackage{stackrel}
\usepackage{graphicx}
\newcommand{\ie}{\emph{i.e.}}

\newcommand{\cf}{\emph{cf.}}
\newcommand{\Real}{\mathbb{R}}

\newcommand{\supp}{\mathop{\mathrm{supp}}\nolimits}
\newcommand{\Dom}{\mathsf{D}}

\newcommand{\Hilbert}{\mathcal{H}}
\newcommand{\eps}{\varepsilon}
\newcommand{\sii}{L^2}

\newcommand{\der}{\mathrm{d}}
\newtheorem{Theorem}{Theorem}
\newtheorem{Lemma}{Lemma}

\theoremstyle{definition}
\newtheorem{Remark}{Remark}

\usepackage[normalem]{ulem}

\definecolor{DarkGreen}{rgb}{0,0.5,0.1} 

\newcommand\soutD{\bgroup\markoverwith
{\textcolor{DarkGreen}{\rule[.5ex]{2pt}{1pt}}}\ULon}
\newcommand\soutP{\bgroup\markoverwith
{\textcolor{blue}{\rule[.5ex]{2pt}{1pt}}}\ULon}
\newcommand{\Hm}[1]{\leavevmode{\marginpar{\tiny%
$\hbox to 0mm{\hspace*{-0.5mm}$\leftarrow$\hss}%
\vcenter{\vrule depth 0.1mm height 0.1mm width \the\marginparwidth}%
\hbox to
0mm{\hss$\rightarrow$\hspace*{-0.5mm}}$\\\relax\raggedright #1}}}

\begin{document}
%
\title{\textbf{\Large
Ruled strips with asymptotically diverging twisting
}}
\author{
David Krej\v{c}i\v{r}{\'\i}k\footnotemark
\quad and \
Rafael Tiedra de Aldecoa\footnotemark%
}
\date{\small 
\emph{
\begin{quote}
\begin{itemize}
\item[$\dag$] 
Department of Mathematics, Faculty of Nuclear Sciences and 
Physical Engineering, Czech Technical University in Prague, 
Trojanova 13, 12000 Prague 2, Czechia;
david.krejcirik@fjfi.cvut.cz.%
\item[$\ddag$] 
Facultad de Matem\'aticas,
Pontificia Universidad Cat\'olica de Chile,
Av. Vicu\~na Mackenna 4860, Santiago de Chile;
rtiedra@mat.uc.cl.
\end{itemize}
\end{quote}
}
\smallskip
1 February 2018}
\maketitle
%
\footnotetext[2]{Supported by the GACR grant No.\ 18-08835S 
and by FCT (Portugal) through project PTDC/MAT-CAL/4334/2014.} 
\footnotetext[3]{Supported by the Chilean Fondecyt Grant 1170008.}

\begin{abstract}
\noindent
We consider the Dirichlet Laplacian in a two-dimensional strip composed of segments
translated along a straight line with respect to a rotation angle with velocity
diverging at infinity. We show that this model exhibits a ``raise of
dimension'' at infinity leading to an essential spectrum determined by an
asymptotic three-dimensional tube of annular cross-section. If the cross-section of
the asymptotic tube is a disk, we also prove the existence of discrete eigenvalues below
the essential spectrum.
\end{abstract}

\medskip
\noindent
{\small
\begin{tabular}{ll}
\textbf{MSC (2010):} & \rule{-2ex}{0ex}
35P15, 58C40, 58J50, 81Q10, 81Q35.
\\
\textbf{Keywords:} & \rule{-4.5ex}{0ex}
Laplace-Beltrami operator, Dirichlet condition, spectrum, 
infinite strip, ruled surface, twisting.
\end{tabular}
}
%

\section{Introduction} 
%

\subsection{Motivation and context}
The Laplacian arises in many areas of natural sciences as the leading
contribution in the stationary representation of the partial differential equations
determining the time evolution. 
The solution of the time-dependent problem
can thus be inferred from the spectral analysis of the Laplacian, but the latter is also
important on its own for spectral quantities have typically direct physical
interpretations. Moreover, the structure of the underlying manifold leads to an
interesting interaction of analysis and geometry.

An enormous amount of work has been conducted over the last century
to understand spectral-geometric properties of the Laplacian on Riemannian manifolds.
The situation is probably best understood for the two extreme cases of complete
non-compact and compact manifolds, 
while non-complete non-compact problems,
where the Laplacian may possess 
both the essential and discrete spectra,
are typically much less
studied. A prominent exception is provided by \emph{tubes}, \ie\
tubular neighbourhoods of submanifolds, 
whose spectrum is relatively well understood
due to an intensive study of various quantum models 
motivated by the advent of nanotechnology. 
For unbounded tubes embedded in the Euclidean space 
and Dirichlet boundary conditions,
typical spectral properties of the Laplacian 
can be roughly summarised as follows:

\begin{enumerate}
\item
If the principal curvatures of a complete non-compact
submanifold $\Sigma\subset\Real^{\mathsf{dim}+\mathsf{codim}}$  
of dimension~$\mathsf{dim}$
vanish at infinity and the transport of the tube
cross-section $\omega\subset\Real^{\mathsf{codim}}$ along~$\Sigma$ is asymptotically
parallel, then the essential spectrum of the Dirichlet Laplacian in the tube coincides
with the half-line $[E_1,\infty)$, where~$E_1$ is the lowest
Dirichlet eigenvalue of~$\omega$: 
see \cite{KL}.  
\item
In the precedent situation,
if the transport of~$\omega$ along~$\Sigma$ is parallel,
then the Dirichlet Laplacian in the tube always possesses
discrete eigenvalues below~$E_1$
whenever~$\Sigma$ is parabolic (typically $\mathsf{dim} \le 2$) 
and non-trivially curved: 
see \cite{DE,ChDFK} for $\mathsf{dim}=1$ and
\cite{DEK2,CEK,LL1,LL3,Lu-Rowlett_2012} for $\mathsf{codim}=1$.
\item
On the other hand, 
if~$\Sigma$ is totally geodesic
and the transport of~$\omega$ along~$\Sigma$ is not parallel 
(in which case we say that the tube is \emph{twisted}), 
then the spectrum of the Dirichlet Laplacian is 
purely essential and Hardy-type inequalities hold: 
see \cite{EKK,KZ1} for $\mathsf{dim}=1$ and $\mathsf{codim}=2$.
\end{enumerate}

Analogous results are also known for tubes embedded 
in a Riemannian manifold~$\mathcal{A}$ 
instead of the Euclidean space:
see \cite{K1,K3,LL2,KK3,Haag-Lampart-Teufel_2015,Lampart-Teufel_2017}.
In the simplest non-trivial situation where $\Sigma$
is a curve in a two-dimensional surface~$\mathcal{A}$
and the cross-section~$\omega$ is a symmetric interval,
any non-trivial curvature of~$\Sigma$ and/or
non-negative Gauss curvature of~$\mathcal{A}$,
both vanishing at infinity in an appropriate sense, 
lead to the existence of discrete spectra (see~\cite{K1}),
while Hardy-type inequalities hold 
if~$\Sigma$ is a geodesic and~$\mathcal{A}$ is
non-positively curved (see~\cite{K3,KK3}).
Hence the negative curvature of~$\mathcal{A}$ acts as twisting,
and in fact the prominent example of ``strips on ruled surfaces''
from~\cite{K3} can be re-considered as ``twisted strips'' in~$\Real^3$.
This latter geometric setting  will be our primary concern in this paper.

In general, it is also known that the essential spectrum 
may change if the curvatures of~$\Sigma$ or~$\mathcal{A}$
do not vanish at infinity 
or the transport of the cross-section~$\omega$ along~$\Sigma$ 
is not asymptotically parallel. In summary, the curvature
of~$\Sigma$ has the tendency to lower the spectrum, while twisting has the opposite
effect: see~\cite{K6-with-erratum}. A striking illustration of the
latter is given in the recent paper~\cite{K11} of the first author in which it is
shown that the spectrum of the Dirichlet Laplacian in a three-dimensional tube in
$\mathcal{A}=\Real^3$ about one-dimensional flat submanifold $\Sigma=\Real$ can even
become purely discrete if the velocity of rotation of 
a two-dimensional non-symmetric cross-section~$\omega$ 
along~$\Sigma$ diverges at infinity. This fact can be seen as
an analogue for three-dimensional tubes of Donnelly's result for complete surfaces
\cite{Donnelly-81} which says that the spectrum of the Laplacian is purely discrete
if the Gauss curvature diverges to minus infinity at infinity.

The present paper was initiated as an attempt to generalise the Euclidean model
of~\cite{K11} to a manifold setting. However, we soon realised that
the features of this generalisation are actually quite different and lead to new
phenomena which we believe
are of interest for the spectral-geometric community. Indeed, the
model that we consider exhibits a sort of \emph{raise of dimension} at infinity, which
is responsible not only for 
a shift of the essential spectrum but also for the
existence of discrete eigenvalues.
 
\subsection{Model and main results}
To keep the model as simple as possible, 
we restrict to the one-dimensional base manifold
$\Sigma := \{(s,0,0) : s\in\Real\}$
that we regard as a submanifold of a ruled surface~$\mathcal{A}$ in~$\Real^3$
parameterised by the mapping
\begin{equation}\label{layer}
  \mathcal{L}: \Real^2 \to \Real^3:
  \left\{
  (s,t)\mapsto\big(s,t\cos\theta(s),t\sin\theta(s)\big)
  \right\}
  ,
\end{equation}
where $\theta: \Real\to\Real$ is a (locally) Lipschitz continuous function.
The Gauss curvature of~$\mathcal{A}$ is given by
$$
  K(s,t) = - \frac{\theta'(s)^2}{\big[1+\theta'(s)^2\,t^2\big]^2}\,,
$$
and in agreement  with the general result~\cite[Prop.~3.7.5]{Kli}
we observe that it is everywhere non-positive.
For the tube cross-section, we set $\omega := (a_1,a_2)$,
where $a_1 < a_2$ are two real numbers.
In the terminology of the previous subsection: $\mathsf{dim}=\mathsf{codim}=1$;
$\Sigma$ is parabolic and totally geodesic;
and~$\omega$ is symmetric or non-symmetric. The tube is defined as the image 
$$
  \Omega := \mathcal{L}\big(\Real\times(a_1,a_2)\big)
  \,,
$$
which has a geometrical meaning 
of a \emph{twisted strip} composed of segments $(a_1,a_2)$
translated along the straight line~$\Sigma$ in~$\Real^3$
with respect to the rotation angle~$\theta$:
see Figure~\ref{Fig}.
It is a special case of strips on surfaces studied in \cite{K1,K3,KK3}.

We are interested in the operator $-\Delta_D^\Omega$ in $\sii(\Omega)$
that acts as the Laplace-Beltrami operator in~$\Omega$ 
and satisfies Dirichlet boundary conditions on~$\partial\Omega$;
it is defined in a standard way as the Friedrichs extension 
of the differential expression initially defined on compactly supported 
smooth functions in~$\Omega$, see Section~\ref{Sec.Fermi} below.
Our primary interest is to investigate spectral properties of $-\Delta_D^\Omega$
in the regime of diverging twisting, \ie,
\begin{equation}\label{diverge}
  \lim_{|s| \to \infty} |\theta'(s)| = \infty \,.
\end{equation}
In this case, the Gauss curvature $K(s,t)$ tends to zero as $|s| \to \infty$
for every non-zero~$t$, but the fact that the transport
of~$\omega$ along~$\Sigma$ (when the latter is regarded as embedded in~$\Real^3$)
is far from being parallel (in which case $\theta'=0$)
leads to peculiar spectral properties. In fact, looking at Figure~\ref{Fig},
one can convince oneself that the two-dimensional strip~$\Omega$ 
actually looks at infinity like 
a three-dimensional tube of annular cross-section
\begin{equation}\label{annulus}
  A_{r_1,r_2} := \{x\in\Real^2 : r_1 < |x| < r_2\}
  \,, 
\end{equation}
where
\begin{equation}\label{radii} 
  r_1 := 
  \begin{cases}
    \min\{|a_1|,|a_2|\} & \mbox{if}\quad a_1 a_2 \ge 0 \,,
    \\
    0 & \mbox{if}\quad a_1 a_2 \le 0 \,,  
  \end{cases}
  \qquad
  r_2 := \max\{|a_1|,|a_2|\}
  \,.
\end{equation}
\begin{figure}[h!]
\fbox{
\begin{minipage}{\textwidth}
\begin{center}
\bigskip
\includegraphics[width=0.9\textwidth]{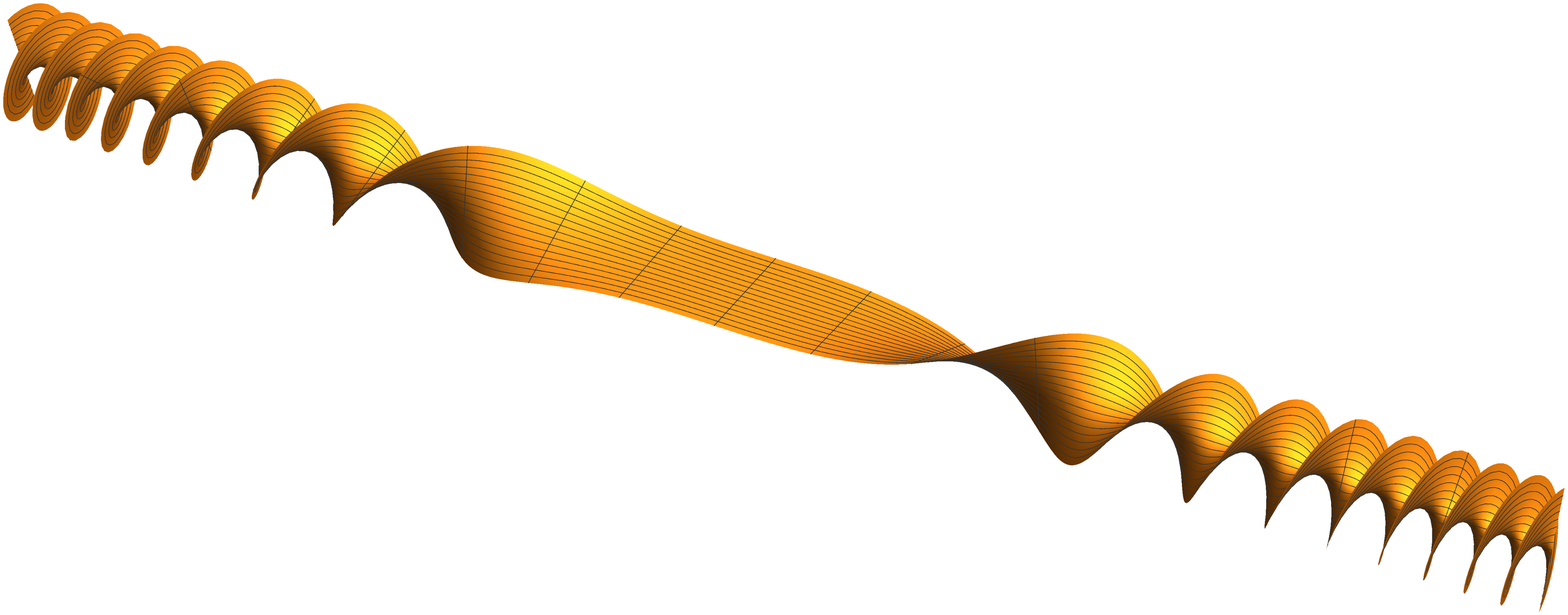}
\bigskip\\
\includegraphics[width=0.9\textwidth]{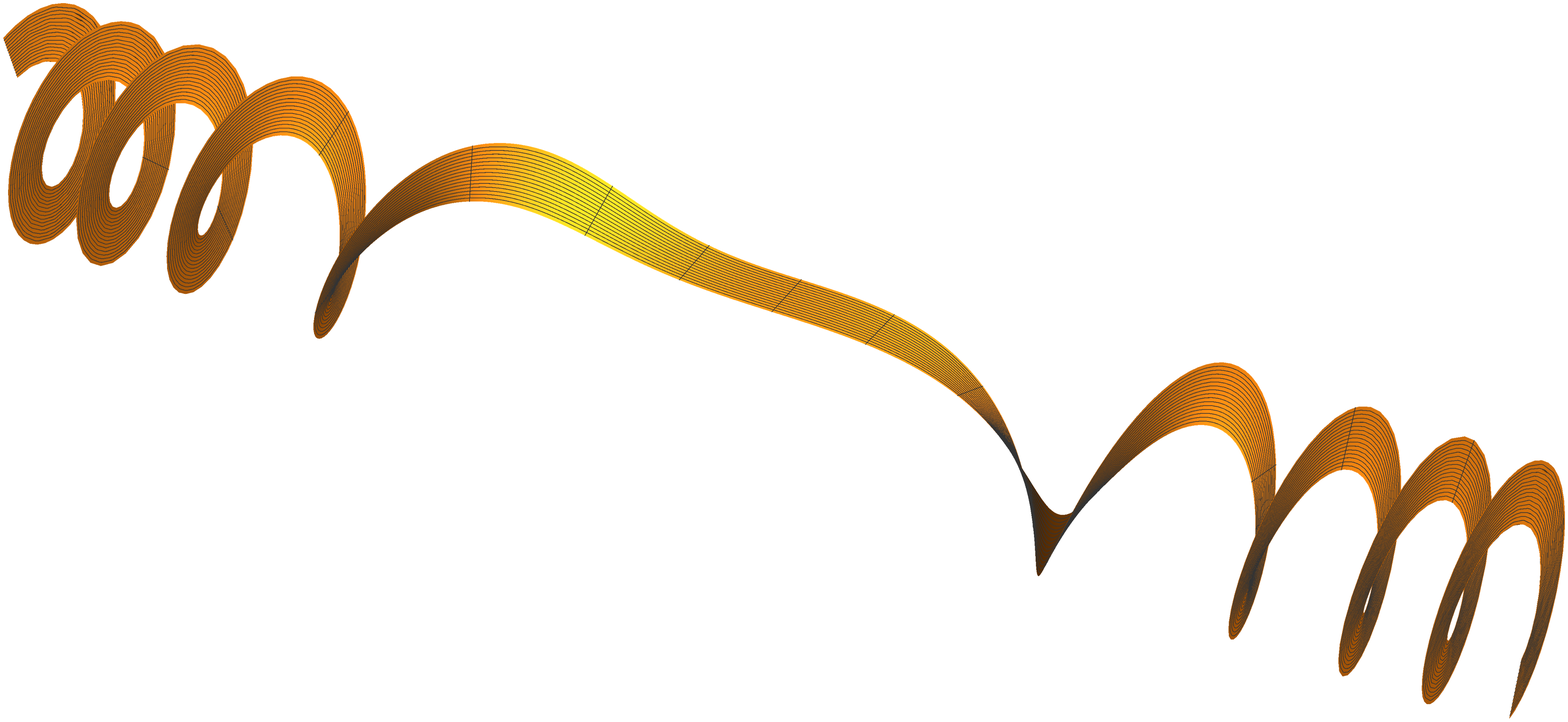}
\caption{\small Ruled strips with asymptotically diverging twisting
(case $a_1 a_2 \le 0$ above, and case $a_1 a_2 > 0$ below).}
\label{Fig}
\end{center}
\end{minipage}
}
\end{figure}

Our first main result about the essential spectrum shows 
that this intuition is correct:
\begin{Theorem}\label{Thm.ess}
If~\eqref{diverge} holds, then
\begin{equation}\label{ess} 
  \sigma_\mathrm{ess}(-\Delta_D^\Omega) = [\lambda_1,\infty)\,,
\end{equation}
where $\lambda_1$ denotes the lowest eigenvalue of the Dirichlet
Laplacian in~$A_{r_1,r_2}$.
\end{Theorem}

This theorem puts into evidence a difference with respect 
to twisted strips with asymptotically vanishing twisting,
\ie\ $\theta'(s) \to 0$ as $|s| \to \infty$,
in which case it is known that
$$
\inf\sigma_\mathrm{ess}(-\Delta_D^\Omega) = [E_1,\infty) \,,
$$
where $E_1:=\pi^2/(a_2-a_1)^2$ is the lowest eigenvalue  
of the Dirichlet Laplacian in the cross-section $\omega=(a_1,a_2)$.
Depending on values of~$a_1$ and~$a_2$,
the number~$\lambda_1$ may be either below or above~$E_1$
(and by continuity there are also~$a_1$ and~$a_2$
for which it equals~$E_1$),
but it is always bounded from above by 
$E_2 := (2\pi)^2/(a_2-a_1)^2$,
the second eigenvalue of the Dirichlet Laplacian in~$\omega$.
Hence the present situation
substantially differs from the case of 
three-dimensional tubes with asymptotically diverging twisting
studied in~\cite{K11};
in view of~\eqref{ess}, there is always some essential
spectrum in twisted strips, 
while it is shown in~\cite{K11} that the spectrum is purely discrete
for tubes satisfying~\eqref{diverge} 
in a regime analogous to $a_1 a_2 > 0$.
Also, technically the proof of Theorem~\ref{Thm.ess}
differs from the standard proofs for asymptotically flat geometries
as well as from the demonstration given in~\cite{K11}.

Our second main result shows that there is always spectrum  
below~$\lambda_1$ provided that the cross-section~$\omega$ 
is twisted with respect to a point \emph{inside} 
the interval:
\begin{Theorem}\label{Thm.disc.intro}
If~\eqref{diverge} and $a_1 a_2 \le 0$ hold, 
then 
$$
  \sigma_\mathrm{disc}(-\Delta_D^\Omega) \cap (0,\lambda_1) \not= \varnothing
  \,.
$$
\end{Theorem}

To show that discrete eigenvalues of~$-\Delta_D^\Omega$ 
also exist in some cases $a_1 a_2 >0$ remains an open problem.
On the other hand, we conjecture that the discrete spectrum
is empty whenever $a_1 a_2$ is large enough.

Finally, we leave as an open problem the study of the nature
of the essential spectrum located in Theorem~\ref{Thm.ess}
(possible existence and location 
or absence of embedded eigenvalues and the absence of
singular continuous spectrum).
In the case of twisting vanishing at infinity,
this study was performed in~\cite{KTdA} 
with help of Mourre theory.
In the present regime~\eqref{diverge}, however,
the choice of the conjugate operator is not clear.

The paper is organised as follows:
Section~\ref{Sec.pre} is devoted to necessary prerequisites,
in particular the definition of the Laplacian $-\Delta_D^\Omega$,
and Theorems~\ref{Thm.ess} and~\ref{Thm.disc.intro}
are proved in Sections~\ref{Sec.ess} and~\ref{Sec.disc}, respectively.

\section{Preliminaries}\label{Sec.pre}
%
Let us start with general preliminaries
valid for an arbitrary twisting angle~$\theta$.	

\subsection{Fermi coordinates}\label{Sec.Fermi} 
Using the (Fermi) coordinates $(s,t)$ defined by~\eqref{layer},
$\Omega$~can be identified with the Riemannian manifold 
$\big(\Real\times(a_1,a_2), G\big)$,
where~$G$ is the metric induced by $\mathcal{L}$,
\begin{equation}\label{metric}
  G := \nabla\mathcal{L} \cdot (\nabla\mathcal{L})^\top =
  \begin{pmatrix}
    f^2 & 0 \\
    0 & 1
  \end{pmatrix}
  \qquad\mbox{with}\qquad
  f(s,t) := \sqrt{1+\theta'(s)^2\,t^2}
  \,.
\end{equation}
The function~$f$ is the Jacobian of~$\mathcal{L}$. 
In these coordinates, $-\Delta_D^\Omega$ can be identified 
with the self-adjoint operator~$H$ in the Hilbert space
$$
  \Hilbert := \sii\big(\Real\times(a_1,a_2),f(s,t)\,\der s \, \der t\big)
$$
associated with the closure~$h$ of the form
$$
  \dot{h}[\Psi] := \big(\partial_i\Psi,G^{ij}\partial_j\Psi\big)_{\Hilbert}
  = \int f^{-1} \, |\partial_s\Psi|^2 
  + \int f \, |\partial_t\Psi|^2
  \,, \qquad \Dom(\dot{h}) := C_0^\infty\big(\Real\times(a_1,a_2)\big)\,.
$$
Here $\int$ denotes the integration over $\Real \times (a_1,a_2)$
and the arguments of the functions are suppressed for brevity.
Furthermore, we adopt the standard notation~$G^{ij}$ 
for the coefficients of the inverse metric~$G^{-1}$
and the Einstein summation convention is used with the range 
of indices being $1,2$. 
Throughout the paper, the first and second variables 
are consistently denoted by~$s$ and~$t$, respectively.
 
By definition, we have
$$
  \Dom(h)
  = \overline{C_0^\infty\big(\Real\times(a_1,a_2)\big)}^{\|\cdot\|_{\Hilbert_1}}
  =: W^{1,2}\big(\Real\times(a_1,a_2),G\big)
$$
with the Sobolev-type norm $\|\cdot\|_{\Hilbert_1}$ given by
\begin{equation}\label{h-norm}
  \|\Psi\|_{\Hilbert_1} 
  := 
  \sqrt{ \dot{h}[\Psi]
  + \|\Psi\|_\Hilbert^2}
  =
  \sqrt{
  \int f^{-1} \, |\partial_s\Psi|^2 
  + \int f \, |\partial_t\Psi|^2
  + \int f \, |\Psi|^2
  }
  \,.
\end{equation}
In a distributional sense, $H$~acts as
\begin{equation}\label{LB}
  H = - |G|^{-1/2} \partial_i |G|^{1/2} G^{ij} \partial_j
  = - f^{-1} \partial_s f^{-1} \partial_s
  - f^{-1} \partial_t f \partial_t
  \qquad\mbox{with}\qquad
  |G| := \det(G) = f^2.
\end{equation}
\begin{Remark}\label{Rem.unitary}
For the purpose of this remark only,
let us assume for a moment extra hypotheses 
about the differentiability of~$\theta$,
say  
\begin{equation}\label{Ass.extra}
  \theta'', \, \theta''' \in L_\mathrm{loc}^\infty(\Real) 
  \,.
\end{equation}
Then~$H$ is unitarily equivalent to a Schr\"odinger-type operator 
(see \cite[Sec.~3.2]{KTdA} or \cite[Sec.~4.4]{FK4})
$$
  \hat{H} = -\partial_s f^{-2} \partial_s - \partial_t^2 + V
  \qquad \mbox{in} \qquad
  \sii\big(\Real\times(a_1,a_2)\big) 
  \,,
$$
with potential 
$$
  V := V_1 + V_2 
  \,, \qquad
  V_1 := -\frac{5}{4} \frac{(\partial_s f)^2}{f^4} 
  + \frac{1}{2} \frac{\partial_s^2 f}{f^3}
  \,, \qquad
  V_2 := -\frac{1}{4} \frac{(\partial_t f)^2}{f^2}
  + \frac{1}{2} \frac{\partial_t^2 f}{f}
  \,.
$$
More specifically, $\hat{H}$ is the operator associated 
with the closure~$\hat{h}$ of the form
$$
  \dot{\hat{h}}[\Psi] := 
  \int f^{-2} \, |\partial_s \Psi|^2 
  + \int|\partial_t \Psi|^2  
  + \int V \, |\Psi|^2
  \,, \qquad
  \Dom(\dot{\hat{h}}) := C_0^\infty\big(\Real\times(a_1,a_2)\big)
  \,.
$$
Using the special form of~$f$ given in~\eqref{metric}, 
we obtain
\begin{align}
  V_1(s,t) &= -\frac{7}{4}
  \frac{\theta'(s)^2 \, \theta''(s)^2 \, t^4}
  {\big[1+\theta'(s)^2\,t^2\big]^3}
  + \frac{1}{2}
  \frac{\theta''(s)^2 \, t^2}
  {\big[1+\theta'(s)^2\,t^2\big]^2}
  + \frac{1}{2}
  \frac{\theta'(s) \, \theta'''(s) \, t^3}
  {\big[1+\theta'(s)^2\,t^2\big]^2}
  \,, 
  \label{V1}
  \\
  V_2(s,t) &= \frac{\theta'(s)^2 \,
  \big[2-\theta'(s)^2 \, t^2\big]}
  {4\,\big[1+\theta'(s)^2\,t^2\big]^2}
  \,.
  \label{V2}
\end{align}
Note that~\eqref{V2} makes sense 
under the present minimal (Lipschitz)
regularity assumptions on~$\theta$.

In the present work,
in order to avoid the additional regularity hypotheses,
we do not use the unitarily equivalent operator~$\hat{H}$ 
to establish spectral properties of~$H$.
However, $\hat{H}$~will be occasionally recalled 
to get some insights into the problem.
\end{Remark}

\subsection{A lower bound}
We follow the method of~\cite{K3} (see also \cite[Sec.~7.2]{KK3})
to get a convenient lower bound to~$H$.
In the sense of quadratic forms in~$\Hilbert$, we have  
\begin{equation}\label{l.bound}
  H \ge - f^{-1} \partial_s f^{-1} \partial_s + \lambda
  \,,
\end{equation}
where, for almost every fixed $s \in \Real$, 
$\lambda(s)$ denotes the lowest eigenvalue of the ``transverse'' operator 
\begin{equation}\label{transverse}
  L_s := - f(s,t)^{-1} \, \frac{\der}{\der t} \, f(s,t) \, \frac{\der}{\der t}
  \qquad \mbox{in} \qquad
  \Hilbert_s := \sii\big((a_1,a_2),f(s,t)\,\der t\big)
  \,,
\end{equation}
subject to Dirichlet boundary conditions at~$a_1$ and~$a_2$.
(With an abuse of notation, we denote by the same symbol~$\lambda$
both the function on~$\Real$ and $\lambda \otimes 1$ on $\Real \times (a_1,a_2)$.)
More precisely, $L_s$~is the operator associated with 
the closure $\ell_s$ of the form
$$
  \dot{\ell}_s[\psi] := \int_{a_1}^{a_2} |\psi'(t)|^2 \, f(s,t)\,\der t
  \,, \qquad 
  \Dom(\dot{\ell}_s) := C_0^\infty\big((a_1,a_2)\big)
  \,,
$$
in the Hilbert space~$\Hilbert_s$.
We have
\begin{equation}\label{lambda}
  \lambda(s)
  = \inf_{\stackrel[\psi\not=0]{}{\psi \in C_0^\infty((a_1,a_2))}}
  \,
  \frac{\displaystyle \int_{a_1}^{a_2} |\psi'(t)|^2 \, f(s,t) \, \der t}
  {\displaystyle \int_{a_1}^{a_2} |\psi(t)|^2 \, f(s,t) \, \der t}
  \,.
\end{equation}
\begin{Remark} 
The operator~$L_s$ is unitarily equivalent to 
\begin{equation}\label{unitarily}
  \hat{L}_s := -\frac{\der^2}{\der t^2} + V_2(s,t) 
  \qquad \mbox{in} \qquad
  \sii\big((a_1,a_2)\big)
  \,,
\end{equation}
where~$V_2$ is defined in~\eqref{V2}.
More precisely, $\hat{L}_s$~is the operator associated with 
the closure $\hat{\ell}_s$ of the form
$$
  \dot{\hat{\ell}}_s[\psi] := \int_{a_1}^{a_2} |\psi'(t)|^2 \,\der t
  + \int_{a_1}^{a_2} V_2(s,t) \, |\psi(t)|^2 \,\der t
  \,, \qquad 
  \Dom(\dot{\hat{\ell}}_s) := C_0^\infty\big((a_1,a_2)\big)
  \,.
$$
Consequently, we also have
\begin{equation}\label{lambda.bis}
  \lambda(s)
  = \inf_{\stackrel[\psi\not=0]{}{\psi \in C_0^\infty((a_1,a_2))}}
  \,
  \frac{\displaystyle \int_{a_1}^{a_2} |\psi'(t)|^2 \, \der t
  + \int_{a_1}^{a_2} V_2(s,t) \, |\psi(t)|^2 \,\der t}
  {\displaystyle \int_{a_1}^{a_2} |\psi(t)|^2 \, \der t}
  \,.
\end{equation}

\end{Remark}

\subsection{An effective operator}\label{Sec.eff}
Now let us look at the case of asymptotically diverging twisting.
Assuming~\eqref{diverge}, we already saw 
that the strip~$\Omega$ looks at infinity like 
a three-dimensional tube of annular cross-section~\eqref{annulus}.
If~$r_1$ is positive (\ie~$a_1 a_2 > 0$), 
$A_{r_1,r_2}$ is an annulus of radii~$r_1$ and~$r_2$,  
while in the case $r_1=0$ (\ie~$a_1 a_2 \le 0$) 
$A_{0,r_2}$~can be identified with a disk of radius~$r_2$
(which we understand as a degenerate annulus,
to keep the same terminology). By the rotational symmetry, 
the lowest Dirichlet eigenvalue of~$A_{r_1,r_2}$ satisfies
\begin{equation}\label{lambda1}
  \lambda_1
  = \inf_{\stackrel[\psi\not=0]{}{\psi \in C_0^\infty((r_1,r_2))}}
  \,
  \frac{\displaystyle \int_{r_1}^{r_2} |\psi'(r)|^2 \, r \, \der r}
  {\displaystyle \int_{r_1}^{r_2} |\psi(r)|^2 \, r \, \der r}
  \,.
\end{equation}

To further support the above geometric intuition,
let us notice that the Jacobian $f(s,t)$ 
asymptotically decouples as follows
\begin{equation}\label{decoupling} 
  \lim_{|s|\to\infty} 
  \frac{f(s,t)}{f_\infty(s,t)} = 1 
  \qquad \mbox{with} \qquad
  f_\infty(s,t) := |\theta'(s)|\,|t|
\end{equation}
for every non-zero~$t$ (that is, for all $t \in (a_1,a_2)$ if $a_1 a_2 > 0$).
It is therefore natural to consider the following
asymptotic version of the operator~$L_s$ introduced in~\eqref{transverse}:  
\begin{equation}\label{L.infinity} 
  L_\infty := - |t|^{-1} \, \frac{\der}{\der t} \, |t| \, \frac{\der}{\der t}
  \qquad \mbox{in} \qquad
  \Hilbert_\infty := \sii\big((a_1,a_2),|t|\,\der t\big)
  \,,
\end{equation}
subject to Dirichlet boundary conditions at~$a_1$ and~$a_2$.
More precisely, $L_\infty$~is the operator associated with 
the closure $\ell_\infty$ of the form 
$$
  \dot{\ell}_\infty[\psi] := \int_{a_1}^{a_2} |\psi'(t)|^2 \, |t|\,\der t
  \,, \qquad 
  \Dom(\dot{\ell}_\infty) := C_0^\infty\big((a_1,a_2)\big)
  \,,
$$
in the Hilbert space~$\Hilbert_\infty$.
It is easy to check that the lowest eigenvalue of~$L_\infty$
coincides with~\eqref{lambda1}, \ie   
\begin{equation}\label{lambda1.bis}
  \lambda_1
  = \inf_{\stackrel[\psi\not=0]{}{\psi \in C_0^\infty((a_1,a_2))}}
  \,
  \frac{\displaystyle \int_{a_1}^{a_2} |\psi'(t)|^2 \, |t| \, \der t}
  {\displaystyle \int_{a_1}^{a_2} |\psi(t)|^2 \, |t| \, \der t}
  \,.
\end{equation}
The claim is indeed obvious for $a_1 a_2 \ge 0$.

If $a_1 a_2 \le 0$, the lowest eigenvalue of~$L_\infty$ is actually determined 
by the minimum between the lowest eigenvalues of the Dirichlet Laplacian
in disjoint disks of radii~$|a_1|$ and~$a_2$.
More specifically, if $a_1 a_2 < 0$, 
the eigenfunction~$\psi_1$ of~$L_\infty$
corresponding to~$\lambda_1$ satisfies the boundary value problem
\begin{equation}\label{Neumann}
\left\{
\begin{aligned}
  - t^{-1} \big(t \,\psi_1'(t)\big)' &= \lambda_1 \psi_1(t) \,, 
  && t \in (a_1,0) \cup (0,a_2) \,, 
  \\
  \psi_1(a_1) = \psi_1(a_2) &=0 \,,
  \\
  \psi_1'(0^-) = \psi_1'(0^+) &=0 \,.
\end{aligned}
\right.
\end{equation}
Hence~$\psi_1$ is not necessarily continuous at~$0$
and the problem is decoupled into the disjoint intervals 
$(a_1,0)$ and $(0,a_2)$ via the extra Neumann boundary condition.  

Standard arguments show that $\psi_1$ is infinitely smooth 
in $(a_1,a_2)\setminus\{0\}$
and can be chosen non-negative.
In fact, $\psi_1$ is then positive if $a_1 a_2 \ge 0$.	 
If $a_1 a_2 < 0$, however, $\psi_1$ is positive only in one 
of the intervals $(a_1,0)$ or $(0,a_2)$,
while it is necessarily zero in the other 
(namely in the smaller).
The corresponding eigenvalue~$\lambda_1$ is always simple 
unless $|a_1| = a_2$, in which case it is doubly degenerate
(but it is of course simple as an eigenvalue of 
the Dirichlet Laplacian in $A_{0,r_2}$). 

\begin{Remark}
In particular, if $r_1=0$ (\ie~$a_1 a_2 \le 0$), 
then
\begin{equation}\label{lambda1.sym}
  \lambda_1 = 
  \left(\frac{j_{0,1}}{r_2}\right)^2 
  \,,
\end{equation}
where $j_{0,1} \approx 2.4$ 
is the first positive zero of the Bessel function~$J_0$.
If moreover $r_2=|a_1|$ (\ie~$|a_1| \ge a_2$), then
\begin{equation}\label{psi1.sym}
  \psi_1(t) = 
  \begin{cases}
    J_0\big(\sqrt{\lambda_1} t\big) 
    & \mbox{if} \quad t\in(a_1,0) \,,
    \\
    0 
    & \mbox{if} \quad t\in(0,a_2) \,.
  \end{cases}
\end{equation}
\end{Remark}

Let us denote by $\{\lambda_k\}_{k=1}^\infty$ and $\{\psi_k\}_{k=1}^\infty$
the set of eigenvalues of~$L_\infty$
and corresponding eigenfunctions, respectively.
We choose~$\psi_k$ real-valued and normalised to~$1$ in~$\Hilbert_\infty$.
The boundary value problem~\eqref{Neumann}
holds for each couple $(\lambda_k,\psi_k)$ instead of $(\lambda_1,\psi_1)$.
Note that the set $\{\lambda_k\}_{k=1}^\infty$ coincides 
with the eigenvalues of the Dirichlet Laplacian in the annulus~$A_{r_1,r_2}$
associated with radially symmetric eigenfunctions.

\section{The essential spectrum}\label{Sec.ess}
%
Let 
$$ 
  E_k := \left(\frac{k\pi}{a_2-a_1}\right)^2
  \,, \qquad 
  k=1,2,\dots
  \,,
$$
denote the eigenvalues of the Dirichlet Laplacian in $\sii\big((a_1,a_2)\big)$.
The corresponding normalised eigenfunctions are given by
\begin{equation}\label{chi} 
  \chi_k(t) := 
  \sqrt{\frac{2}{a_2-a_1}} \, \sin\big(\sqrt{E_k}\,(t-a_1)\big)
  \,.
\end{equation}
It was shown in~\cite[Thm.~5.1]{KK3} that
\begin{equation}\label{vanish}
  \lim_{|s| \to \infty} |\theta'(s)| = 0
  \qquad \Longrightarrow \qquad
  \sigma_\mathrm{ess}(H) = [E_1,\infty) 
  \,.
\end{equation}
This result is expected because under the condition
the strip~$\Omega$ looks at infinity like a straight strip of width~$a_2-a_1$
and the spectrum of the latter coincides with the interval $[E_1,\infty)$.
 
Our goal is to show that the condition~\eqref{diverge}
actually shifts the essential spectrum
in the sense of Theorem~\ref{Thm.ess}. 
This result is also intuitively clear because, 
under condition~\eqref{diverge},
the strip~$\Omega$ looks at infinity like 
a three-dimensional tube of cross-section~$A_{r_1,r_2}$,
and~$\lambda_1$ is the lowest eigenvalue of
the Dirichlet Laplacian in~$A_{r_1,r_2}$,
see Section~\ref{Sec.eff}. However, giving a rigorous proof
of this result requires some work, as we shall see.

Using~$\chi_1$ with $a_1:=r_1$ and $a_2:=r_2$ 
as a test function in~\eqref{lambda1}, 
we obtain the upper bound
\begin{equation}\label{upper}
  \lambda_1
  < \left(\frac{\pi}{r_2-r_1}\right)^2 
\end{equation}
(the strict inequality follows from the fact that the chosen test function
differs from the eigenfunction~$\psi_1$).
If $a_1 a_2 \ge 0$, then $r_2-r_1 = a_2-a_1$,
so~\eqref{upper} shows that $\lambda_1 < E_1$ in this case.
On the other hand, if the interval $(a_1,a_2)$ is symmetric,
\ie~$a_2=-a_1$ (a particular case of $a_1 a_2 \le 0$),
it follows from~\eqref{lambda1.sym} that $E_1 < \lambda_1 < E_2$.
If $a_1 a_2 \le 0$ without any further restriction on~$a_1$ and~$a_2$, 
we have  $r_2-r_1 = r_2 \ge (a_2-a_1)/2$,
so~\eqref{upper} together with this crude bound 
shows only that $\lambda_1 < E_2$.
Summing up, we see that, depending on values of~$a_1$ and~$a_2$,
$\lambda_1$ may be either below or above~$E_1$
(and by continuity there are also~$a_1$ and~$a_2$
for which it equals~$E_1$),
but it is always bounded from above by~$E_2$.

Theorem~\ref{Thm.ess} is established in two steps.
First, we establish a lower bound to the threshold 
of the essential spectrum:
\begin{Lemma}\label{Lem.ess2}
If~\eqref{diverge} holds, then 
$
  \inf\sigma_\mathrm{ess}(H) \ge \lambda_1
$. 
\end{Lemma}
\begin{proof}
By imposing an extra Neumann condition at the segments $\{s=\pm s_0\}$ 
with $s_0>0$,
we obtain
$$
  H \ge H_\mathrm{int}^N \oplus H_\mathrm{ext}^N
  \,,
$$
where~$H_\mathrm{int}^N$ is the self-adjoint operator in 
$
  \Hilbert_\mathrm{int} :=
  \sii\big((-s_0,s_0)\times(a_1,a_2),f(s,t)\,\der s \, \der t\big)
$
associated with the closure~$h_\mathrm{int}^N$ of the form
$$
  \dot{h}_\mathrm{int}^N[\Psi] 
  := \big(\partial_i\Psi,G^{ij}\partial_j\Psi\big)_{\Hilbert_\mathrm{int}}
  \,, \qquad
  \Dom(\dot{h}_\mathrm{int}^N) := 
  C_0^\infty\big(\Real\times(a_1,a_2)\big)
  \upharpoonright (-s_0,s_0)\times(a_1,a_2)
  \,,
$$
and~$H_\mathrm{ext}^N$ is an analogously defined operator in
$
  \Hilbert_\mathrm{ext} :=
  \sii\big(\Real\setminus[-s_0,s_0]\times(a_1,a_2),f(s,t)\,\der s \, \der t\big)
$.
Since the spectrum of~$H_\mathrm{int}^N$ is purely discrete,
the minimax principle yields
\begin{equation}\label{brack1}
  \inf\sigma_\mathrm{ess}(H) 
  \ge \inf\sigma_\mathrm{ess}(H_\mathrm{ext}^N)  
  \ge \inf\sigma(H_\mathrm{ext}^N)
  \,.
\end{equation}
A lower bound analogous to~\eqref{l.bound} holds for~$H_\mathrm{ext}^N$ as well.
Consequently, neglecting the differential part in~\eqref{l.bound},
we obtain
\begin{equation}\label{brack2}
  \inf\sigma(H_\mathrm{ext}^N)
  \ge \inf_{|s| \ge s_0} \lambda(s)
  \,.
\end{equation}
This estimate together with~\eqref{brack1} yields
$$
  \inf\sigma(H)
  \ge \inf_{|s| \ge s_0} \lambda(s)
  \,.
$$
Since $s_0>0$ is arbitrary, 
the left-hand side is independent of~$s_0$
and the essential spectrum is a closed set, 
it follows that it is enough to show that
(recall~\eqref{lambda} and~\eqref{lambda1.bis})
\begin{equation}\label{crucial}
  \liminf_{|s| \to \infty} \lambda(s) \ge \lambda_1
\end{equation}
in order to conclude with the proof of the lemma.
 
To establish~\eqref{crucial},
since~$\lambda(s)$ (respectively, $\lambda_1$) is the lowest eigenvalue
of the operator~$L_s$ (respectively, $L_\infty$),
one is lead to establishing a sort of convergence of~$L_s$ to~$L_\infty$
as $|s| \to \infty$. For this, let~$\psi_s$ denote an eigenfunction of~$L_s$
corresponding to $\lambda(s) =: \lambda_s$, with $\psi_s$ normalised to~$1$ in~$\Hilbert_s$,
and let us write the eigenvalue equation $L_{s}\psi_{s}=\lambda_{s}\psi_{s}$
in the weak form
\begin{equation}\label{weak} 
  \ell_{s}(\varphi,\psi_{s}) 
  = \lambda_{s} \, (\varphi,\psi_{s})_{\Hilbert_{s}} 
  \,, \qquad
  \forall\varphi\in\Dom(\ell_{s})\,.
\end{equation}
First of all, we remark that $\lambda_s$ is uniformly bounded in $s \in \Real$
due to~\eqref{upper}. 
This can be verified by using a test function in~\eqref{lambda}
supported away from zero (it is only needed if $a_1 a_2 \le 0$)
and estimating $f(s,t) \ge |\theta'(s)|\,|t|$ in the denominator
and $f(s,t) \le 2 |\theta'(s)|\,|t|$ in the numerator
(the second bound is valid for all~$s$ sufficiently large
due to~\eqref{diverge}),
so that the $s$-dependent terms $|\theta'(s)|$ cancel in
the Rayleigh quotient. 
Consequently, putting $\varphi := \psi_{s}$ in~\eqref{weak},
we obtain  
$$ 
  \|\psi_{s}'\|_{\Hilbert_{s}}^2 \le C
  \,,
$$
where~$C$ is a constant independent of~$s$. 
Recalling the definition of~$f$ in~\eqref{metric}, we deduce
\begin{equation}\label{pre1} 
  \int_{a_1}^{a_2} |\psi_{s}'(t)|^2 \, \der t \le C
  \qquad \mbox{and} \qquad
  |\theta'(s)| \int_{a_1}^{a_2} |\psi_{s}'(t)|^2 \, |t| \, \der t \le C
  \,.
\end{equation}
Similarly, from the normalisation of~$\psi_s$, we deduce
\begin{equation}\label{pre2}  
  \int_{a_1}^{a_2} |\psi_{s}(t)|^2 \, \der t \le 1
  \qquad \mbox{and} \qquad
  |\theta'(s)| \int_{a_1}^{a_2} |\psi_{s}(t)|^2 \, |t| \, \der t \le 1
  \,.
\end{equation}
We may conclude from the latter inequalities in~\eqref{pre1} and~\eqref{pre2}
that $\{\phi_{s}\}_{s \in \Real}$ with 
$\phi_s(t) := |\theta'(s)|^{1/2} \psi_s(t)$
is a bounded family of functions in 
the weighted Sobolev space $\Dom(\ell_\infty)$,
while $\{\psi_{s}\}_{s \in \Real}$ converges to zero in~$\Hilbert_\infty$
as $|s| \to \infty$ due to~\eqref{diverge}.
At the same time, from the former inequalities in~\eqref{pre1} and~\eqref{pre2}
we deduce that $\{\psi_{s}\}_{s \in \Real}$ 
is a bounded family of functions in 
the standard Sobolev space~$W^{1,2}\big((a_1,a_2)\big)$.

Consequently, the sets
$\{\psi_{s}\}_{s \in \Real}$ and $\{\phi_{s}\}_{s \in \Real}$
are precompact in the weak topology of the spaces
$W^{1,2}\big((a_1,a_2)\big)$ and $\Dom(\ell_\infty)$, respectively.
Let~$\psi_\infty$ and~$\phi_\infty$ denote weak limit points,
\ie~there exists a real sequence $\{s_j\}_{j=1}^\infty$ 
such that $|s_j| \to \infty$ as $j \to \infty$ 
for which 
\begin{equation}\label{limits}
\begin{aligned}
  \psi_{s_j} \xrightarrow[j\to\infty]{w} \psi_\infty
  & \quad \mbox{in} \quad W^{1,2}\big((a_1,a_2)\big) \,,
  & \qquad 
  \phi_{s_j} \xrightarrow[j\to\infty]{w} \phi_\infty
  & \quad \mbox{in} \quad \Dom(\ell_\infty) \,,
  \\
  \psi_{s_j} \xrightarrow[j\to\infty]{s} \psi_\infty
  & \quad \mbox{in} \quad \sii\big((a_1,a_2)\big) \,,
  &
  \phi_{s_j} \xrightarrow[j\to\infty]{s} \phi_\infty
  & \quad \mbox{in} \quad \Hilbert_\infty \,.
\end{aligned}
\end{equation}
Here the strong limits follow from the compactness of the embeddings
$W^{1,2}\big((a_1,a_2)\big) \hookrightarrow \sii\big((a_1,a_2)\big)$
and $\Dom(\ell_\infty) \hookrightarrow \Hilbert_\infty$.
We claim that 
$$
  \psi_\infty = 0 \qquad \mbox{in} \qquad \sii\big((a_1,a_2)\big).
$$
By virtue of the latter inequality in~\eqref{pre2},
it is clear for the case $a_1 a_2 > 0$ when we can bound $|t| \ge r_1 > 0$.
If $a_1 a_2 \le 0$, we deduce the claim as follows.
On the one hand, for every sequence of positive numbers
$\{\eps_j\}_{j=1}^\infty$ such that $\eps_j \to 0$ as $j \to \infty$,
we have the convergence result 
$$
  \int_{(a_1,a_2) \setminus [-\eps_j,\eps_j]} |\psi_{s_j}(t)|^2 \, \der t
  \ \xrightarrow[j\to\infty]{} \
  \int_{(a_1,a_2)} |\psi_{\infty}(t)|^2 \, \der t
$$
by the monotone convergence theorem.
On the other hand, we have
$$
  \int_{(a_1,a_2) \setminus [-\eps_j,\eps_j]} |\psi_{s_j}(t)|^2 \, \der t
  \le \frac{1}{\eps_j}
  \int_{(a_1,a_2) \setminus [-\eps_j,\eps_j]} |\psi_{s_j}(t)|^2 \, |t| \, \der t
  \le \frac{1}{\eps_j \, |\theta'(s_j)|}
  \int_{(a_1,a_2)} |\phi_{s_j}(t)|^2 \, |t| \, \der t
  \,,
$$
where the right-hand side tends to zero as $j \to \infty$
provided that we choose for instance $\eps_j := |\theta'(s_j)|^{-1/2}$.
Here we use the hypothesis~\eqref{diverge}.
The fact that $\psi_\infty = 0$ implies $\phi_\infty \not= 0$.
Indeed, recalling the definition of~$f$ in~\eqref{metric}, 
we have
$$
  1 = \|\psi_{s_j}\|_{\Hilbert_{s_j}}^2
  \le \int_{a_1}^{a_2} |\psi_{s_j}(t)|^2 \, \der t
  + \int_{a_1}^{a_2} |\phi_{s_j}(t)|^2 \, |t| \, \der t
  \xrightarrow[j\to\infty]{} 
  \int_{a_1}^{a_2} |\phi_{\infty}(t)|^2 \, |t| \, \der t
  \,.
$$

In the eigenvalue equation~\eqref{weak}, 
let us take $\varphi \in C_0^\infty\big((a_1,a_2)\big)$,
the core of both~$\ell_s$ and~$\ell_\infty$,
write~$s_j$ instead of~$s$, multiply by $|\theta'(s_j)|^{1/2}$
and take the limit $j \to \infty$. 
Using~\eqref{limits}, we obtain
\begin{equation}\label{passing}
  \ell_\infty(\varphi,\phi_\infty) 
  = \lambda_\infty \,
  (\varphi,\phi_\infty)_{\Hilbert_\infty} 
  \qquad \mbox{with} \qquad
  \lambda_\infty := \lim_{j\to\infty} \lambda_{s_j}
  \,.
\end{equation}
Since~$\phi_\infty$ is non-zero, it enables us to conclude
that the eigenvalue~$\lambda_{s_j}$ of $L_{s_j}$ 
tends to an eigenvalue of~$L_\infty$ as $j \to \infty$
with eigenfunction~$\phi_\infty$.  
Consequently, $\lambda_\infty \ge \lambda_1$.
Since the same result~\eqref{passing} is obtained 
for \emph{any} weak limit point of $\{\phi_{s}\}_{s \in \Real}$,
we arrive at the desired bound~\eqref{crucial}.

Summing up, combining \eqref{crucial} with~\eqref{brack1}
and~\eqref{brack2} concludes the proof of the lemma.
\end{proof}

Now we turn to the opposite inclusion 
$\sigma_\mathrm{ess}(H) \supset [\lambda_1,\infty)$. 
Employing~\eqref{decoupling} in~\eqref{LB},
it is expected that the essential spectrum of~$H$
will be determined by the asymptotic operator
\begin{equation}\label{op.as}
  H_\infty := 
  - \frac{|\theta'(s)|^{-1} \partial_s |\theta'(s)|^{-1}   \partial_s}{t^2}
  - |t|^{-1} \partial_t |t| \partial_t
  \qquad \mbox{in} \qquad
  \sii\big(\Real\times(a_1,a_2), |\theta'(s)| \, |t| \, \der s \, \der t\big) 
  \,,
\end{equation}
for which we can assume that $\theta'(s) \not= 0$ for all $s \in \Real$
by using~\eqref{diverge} and possibly re-defining~$\theta'$
on a compact set (the latter does not influence the essential spectrum). 
More specifically, $H_\infty$ is the operator associated 
with the closure~$h_\infty$ of the form
$$
  \dot{h}_\infty[\Psi] := 
  \int  f_\infty^{-1} \, |\partial_s \Psi|^2
  + \int f_\infty \, |\partial_t \Psi|^2  
  \,, \qquad
  \Dom(\dot{h}_\infty) 
  := C_0^\infty\big(\Real\times [(a_1,a_2) \setminus \{0\}] \big)
  \,,
$$
with the function~$f_\infty$ defined in~\eqref{decoupling}.
Notice that the one-dimensional operator 
$-|\theta'(s)|^{-1} \partial_s |\theta'(s)|^{-1}\partial_s$ 
in $\sii\big(\Real,|\theta'(s)|\,\der s\big)$,
understood as the Friedrichs extension of the operator
initially defined on the domain $C_0^\infty(\Real)$,
is unitarily equivalent to the standard Laplacian $-\partial_u^2$ 
in $\sii(\Real)$ with the usual domain $W^{2,2}(\Real)$;
indeed, the unitarily equivalence is accomplished 
by the change of variables $u = \int_0^s|\theta'(\sigma)| \, \der\sigma$.
Since the operator $-\partial_u^2$ can be decomposed 
using the Fourier transform, $H_\infty$~is unitarily equivalent to the direct-integral operator
\begin{equation}\label{fiber} 
  H_\infty \cong \int^\oplus_{\Real} 
  L_\infty^m \ \der m
  \qquad \mbox{with} \qquad
  L_\infty^m := 
  - \frac{1}{|t|} \, \frac{\der}{\der t} \, |t| \, \frac{\der}{\der t} 
  + \frac{m^2}{t^2}
  \,.
\end{equation}
More specifically, $L_\infty^m$ is the operator 
in $\sii\big((a_1,a_2),|t|\,\der t\big)$
associated with the closure $\ell_\infty^m$ of the form
$$
  \dot{\ell}_\infty^m[\psi] 
  := \int_{a_1}^{a_2} |\psi'(t)|^2 \, |t|\,\der t
  \,, \qquad 
  \Dom\big(\dot{\ell}_\infty^m\big) 
  := C_0^\infty\big((a_1,a_2)\setminus\{0\}\big)
  \,.
$$
Let us denote by $\{\lambda_k^m\}_{k=1}^\infty$ 
and $\{\psi_k^m\}_{k=1}^\infty$
the set of eigenvalues of~$L_\infty^m$
and corresponding eigenfunctions, respectively,
with $\psi_k^m$ normalised to~$1$
in $\sii\big((a_1,a_2),|t|\,\der t\big)$.
Recalling~\eqref{L.infinity}, 
notice that $\{\lambda_k^0\}_{k=1}^\infty = \{\lambda_k\}_{k=1}^\infty$
and that $\psi_k^0$ can be chosen in such a way that
$\{\psi_k^0\}_{k=1}^\infty = \{\psi_k\}_{k=1}^\infty$.
It follows from~\eqref{fiber} that
$$
  \sigma(H_\infty) 
  = \sigma_\mathrm{ess}(H_\infty) 
  = [\lambda_1,\infty)
  \,.
$$
This is an \emph{a posteriori} evidence supporting \eqref{ess}. 
Here we use~$H_\infty$ as an inspiration to choose 
an appropriate singular sequence to prove the following lemma
with help of a Weyl-type criterion:
\begin{Lemma}\label{Lem.ess1}
If~\eqref{diverge} holds, then
$
  \sigma_\mathrm{ess}(H) \supset [\lambda_1,\infty)
$.
\end{Lemma}
\begin{proof}
Let $\varphi_0 \in C_0^\infty(\Real)$ be a real-valued function
normalised to~$1$ in $\sii(\Real)$ with $\supp\varphi_0 \subset [0,1]$.
For every $n > 0$, we define 
$$
  \varphi_n(s):= n^{-1/2} \,\varphi_0(n^{-1}s-n)
  \,,
$$
so that 
\begin{equation}\label{support}
  \supp\varphi_n = n^2 + n \supp\varphi_0
  \subset [n^2,n^2+n]
  \subset [n^2,(n+1)^2]
\end{equation}
is ``localised at infinity'' for large~$n$.
The normalisation factor is chosen in such a way that 
also $\varphi_n$ is normalised to~$1$ in $\sii(\Real)$.
 
For every $m \in \Real$, we set 
\begin{equation}\label{test}
  \Psi_n^m(s,t) := \varphi_n(s) \, \omega^m(s) \, \psi_1^m(t)
  \qquad \mbox{with} \qquad
  \omega^m(s) :=  e^{i m \int_0^s |\theta'(\sigma)| \, \der \sigma}
  \,.
\end{equation}
(For simplicity, 
below we shall denote by the same symbols 
$\varphi_n$, $\omega^m$ and~$\psi_1^m$ 
functions $\varphi_n \otimes 1$, $\omega^m \otimes 1$
and $1 \otimes \psi_1^m$ on $\Real\times(a_1,a_2)$,
respectively.)
For every $n > 0$ and $m \in \Real$, 
we have $\Psi_n^m\in \Dom(h)$.
Using the pointwise inequality $f \ge 1$ 
and the normalisation of~$\varphi_n$, 
we get
\begin{equation}\label{est0}
  \|\Psi_n^m\|_\Hilbert
  \ge \|\varphi_n\|_{\sii(\Real)} \, \|\psi_1^m\|_{\sii((a_1,a_2))}
  = \|\psi_1^m\|_{\sii((a_1,a_2))} > 0 
  \,.
\end{equation}
Thus we have a positive lower bound to the norm of~$\Psi_n^m$
which is independent of~$n$.
Our goal is to show that, for every $m \in \Real$,
\begin{equation}\label{Weyl.adapted}
  \big\|(H-\lambda_1^m)\Psi_{n}^m\big\|_{\Hilbert_1^*} 
  := \sup_{\stackrel[\phi\not=0]{}{\phi \in C_0^\infty(\Real\times(a_1,a_2))}}
  \frac{| h(\phi,\Psi_{n}^m) 
  - \lambda_1^m \, (\phi,\Psi_{n}^m)_\Hilbert| }
  {\|\phi\|_{\Hilbert_1}}
  \ \xrightarrow[n \to \infty]{} \ 0
  \,,
\end{equation}
where $\|\cdot\|_{\Hilbert_1}$ is the norm introduced in~\eqref{h-norm}.
Since $\Real \ni m \mapsto \lambda_1^m \in [\lambda_1,\infty)$ 
is a continuous function with range $[\lambda_1,\infty)$, 
property~\eqref{Weyl.adapted} ensures that 
$
  [\lambda_1,\infty) \subset \sigma_\mathrm{ess}(H)
$
with help of the Weyl criterion adapted 
to quadratic forms (\cf~\cite[Thm.~5]{KL}). 
Our inspiration for the choice~\eqref{test}
is the \emph{formal} identity
$$
  (H_\infty-\lambda_1^m) \, \omega^m \, \psi_1^m
  = 0
  \,,
$$
which is behind the fact that an analogue of~\eqref{Weyl.adapted}	
holds for~$H_\infty$ instead of~$H$. 

Since $\Real \ni m \mapsto \lambda_1^m \in [\lambda_1,\infty)$ is even
and the essential spectrum is a closed set,
it is enough to show~\eqref{Weyl.adapted} for $m > 0$.  
Note that $[0,\infty) \ni m \mapsto \lambda_1^m \in [\lambda_1,\infty)$
is strictly increasing.
Let us henceforth assume $m > 0$ and write
\begin{equation}\label{split}
  h(\phi,\Psi_{n}^m) - \lambda_1^m \, (\phi,\Psi_{n}^m)_\Hilbert
  = I_1 + I_2 
  \,,
\end{equation}
where
$$
\begin{aligned} 
  I_1 
  &:= \int f_\infty^{-1} \, \overline{\partial_s\phi} \, \partial_s \Psi_n^m 
  + \int f_\infty \, \overline{\partial_t\phi} \, \partial_t \Psi_n^m
  - \lambda_1^m 
  \int f_\infty	 \, \overline{\phi} \, \Psi_n^m 
  \,,
  \\
  I_2 
  &:= \int (f^{-1} - f_\infty^{-1}) \, 
  \overline{\partial_s\phi} \, \partial_s \Psi_n^m 
  + \int (f- f_\infty)\, 
  \overline{\partial_t\phi} \, \partial_t \Psi_n^m
  - \lambda_1^m 
  \int (f - f_\infty)  \, 
  \overline{\phi} \, \Psi_n^m
  \,.
\end{aligned}
$$
If $a_1 a_2 \le 0$, notice that $\psi_k^m(0) = 0$ whenever $m \not= 0$.
Consequently,
\begin{equation}\label{m.positive}
  C_m := \sup_{t \in (a_1,a_2)} \frac{|\psi_1^m(t)|}{|t|}  
\end{equation}
is a finite constant and the integrals above containing~$f_\infty^{-1}$
are well defined.

Integrating by parts and using the eigenvalue equation
that~$\psi_1^m$ satisfies, it is easy to check that
$$
  I_1 = \int f_\infty^{-1} \, \overline{\partial_s\phi}
  \, \varphi_n' \, \omega^m \, \psi_1^m 
  - i m \int |\theta'| \, f_\infty^{-1} \, \overline{\phi}
  \, \varphi_n' \, \omega^m \, \psi_1^m 
  \,.
$$
Consequently,  
$$
\begin{aligned} 
  |I_1| 
  &\le \sqrt{\int f^{-1} \, |\partial_s\phi|^2}
  \sqrt{\int f \, f_\infty^{-2} \, |\varphi_n'|^2 |\psi_1^m|^2}
  + |m| \sqrt{\int f \, |\phi|^2}
  \sqrt{\int |\theta'|^2 \, f^{-1} \, f_\infty^{-2} \, |\varphi_n'|^2 |\psi_1^m|^2}
  \\
  &\le
  \|\phi\|_{\Hilbert_1} \, C_m \, \left(
  \sqrt{\int f \, |\theta'|^{-2} \, |\varphi_n'|^2 }
  + |m| \sqrt{\int f^{-1}  \, |\varphi_n'|^2}
  \right)
  \\
  &\le
  \|\phi\|_{\Hilbert_1} \, C_m \, (1+|m|) \, 
  \|\varphi_n'\|_{\sii(\Real)}
  \,,
\end{aligned}
$$
where the first estimate follows by the Schwarz inequality,
the second bound is due to~\eqref{m.positive}
and the last inequality employs~\eqref{diverge}
and the fact that~$\varphi_n$ is ``localised at infinity''.
Since
$$
  \|\varphi_n'\|_{\sii(\Real)} 
  = n^{-1} \, \|\varphi_0'\|_{\sii(\Real)}
  \,,
$$
we infer that the $I_1$-part of~\eqref{split} verifies~\eqref{Weyl.adapted}.

Now we turn to estimating the individual terms of~$I_2$. 
Similarly as above, 
noticing the identity $f - f_\infty = 1/(f + f_\infty)$
and recalling that~$\varphi_n$ is real-valued,
we estimate
$$
\begin{aligned} 
  \left|
  \int (f^{-1} - f_\infty^{-1}) \, 
  \overline{\partial_s\phi} \, \partial_s \Psi_n^m 
  \right|
  &\le \sqrt{\int f^{-1} \, |\partial_s\phi|^2}
  \sqrt{\int \frac{(|\varphi_n'|^2+m^2 |\theta'|^2 |\varphi_n|^2) |\psi_1^m|^2}
  {f \, f_\infty^2 \, (f_\infty+f)^2}}
  \\
  &\le
  \|\phi\|_{\Hilbert_1} \, C_m \,
  \sqrt{\int \frac{(|\varphi_n'|^2+m^2 |\theta'|^2 |\varphi_n|^2)}
  {f \, |\theta'|^2 \, (f_\infty+f)^2}}
  \,,
\end{aligned}
$$
where the square root on the second line tends to zero as $n\to\infty$,
due to~\eqref{diverge} with help of the dominated convergence theorem. 
At the same time,
$$
\begin{aligned} 
  \left|
  \int (f - f_\infty) \, 
  \overline{\partial_t\phi} \, \partial_t \Psi_n^m 
  \right|
  &\le \sqrt{\int f \, |\partial_t\phi|^2}
  \sqrt{\int \frac{|\varphi_n|^2 \, |(\psi_1^m)'|^2}
  {f \, (f_\infty+f)}}
  \\
  &\le
  \|\phi\|_{\Hilbert_1} \, 
  \left(\sup_{t\in(a_1,a_2)} |(\psi_1^m)'(t)|^2\right) \,
  \sqrt{\int \frac{|\varphi_n|^2}
  {f \, (f_\infty+f)}}
  \,,
\end{aligned}
$$
where the square root on the second line again tends to zero as $n\to\infty$.
Finally,
$$
\begin{aligned} 
  \left|
  \int (f - f_\infty) \, 
  \overline{\phi} \, \Psi_n^m 
  \right|
  &\le \sqrt{\int f \, |\phi|^2}
  \sqrt{\int \frac{|\varphi_n|^2 \, |{\psi_1^m}|^2}
  {f \, (f_\infty+f)}}
  \\
  &\le
  \|\phi\|_{\Hilbert_1} \, 
  \left(\sup_{t\in(a_1,a_2)} |{\psi_1^m}(t)|^2\right) \,
  \sqrt{\int \frac{|\varphi_n|^2}
  {f \, (f_\infty+f)}}
  \,,
\end{aligned}
$$
where the square root on the second line also tends to zero as $n\to\infty$.
Summing up, we infer that the $I_2$-part of~\eqref{split} also
verifies~\eqref{Weyl.adapted}.
\end{proof}

Theorem~\ref{Thm.ess} follows as a consequence of 
Lemmata~\ref{Lem.ess2} and~\ref{Lem.ess1}.

\begin{Remark}\label{Rem.as.op}
Coming back to the unitarily equivalent operator~$\hat{H}$
from Remark~\ref{Rem.unitary}, 
let us assume in addition to~\eqref{diverge} that 
\begin{equation}\label{diverge.extra} 
  \lim_{|s| \to \infty} \frac{\theta''(s)}{\theta'(s)^2} = 0
  \qquad \mbox{and} \qquad
  \lim_{|s| \to \infty} \frac{\theta'''(s)}{\theta'(s)^3} = 0 
  \,.
\end{equation}
Then
\begin{equation}\label{unitary.limits}
  V(s,t) \xrightarrow[|s| \to \infty]{} -\frac{1}{4t^2}
\end{equation}
for all $t \in (a_1,a_2) \setminus \{0\}$ 
(that is, for all $t \in (a_1,a_2)$ if $a_1 a_2 > 0$).
Hence the natural asymptotic counterpart of~$\hat{H}$ reads
$$
  \hat{H}_\infty := 
  - \frac{\partial_s |\theta'(s)|^{-2} \partial_s}{t^2}
  -\partial_t^2-\frac{1}{4t^2}
  \qquad \mbox{in} \qquad
  \sii\big(\Real\times(a_1,a_2)\big) 
  \,.
$$

Let us also remark that the mean curvature of 
the tube~$\Omega$ (when regarded as a submanifold of~$\Real^3$)
is given by
$$
  M(s,t) = 
  - \frac{\theta''(s) \, t}{\big[1+\theta'(s)^2\,t^2\big]^{3/2}}
  \,.
$$
Hence, assuming~\eqref{diverge} 
and the first condition of~\eqref{diverge.extra}, 
both the Gauss and mean curvatures of~$\Omega$
vanish at infinity for all $t\ne0$
(and even uniformly in $t$ if $a_1 a_2 > 0$).
But despite of this ``asymptotic flatness'',
the essential spectrum substantially differs
from the standard situation~\eqref{vanish},
see Theorem~\ref{Thm.ess}.
\end{Remark}
%

\section{The discrete spectrum}\label{Sec.disc}

The following result shows that there is always \emph{some} spectrum of~$H$ 
below~$\lambda_1$ provided that the cross-section $(a_1,a_2)$ 
is twisted with respect to a point \emph{inside} the interval
(irrespectively whether~\eqref{diverge} holds or not). 
\begin{Theorem}\label{Thm.disc}
If $a_1 a_2 \le 0$, then
$$
  \inf\sigma(H) < \lambda_1
  \,.
$$ 
\end{Theorem}
\begin{proof}
As in the proof of Lemma~\ref{Lem.ess1},
we consider a family of test functions of the form
\begin{equation}
  \Psi_n(s,t) := \varphi_n(s) \, \psi_1(t)
  \,, 
\end{equation}
where~$\psi_1$ is the eigenfunction of~$L_\infty$
corresponding to the lowest eigenvalue~$\lambda_1$
(case $m=0$ of~\eqref{test}),
but now the ``longitudinal'' component~$\varphi_n$ is given by
\begin{equation}
  \varphi_n(s) :=
\begin{cases}
  1 & \mbox{if} \quad |s| < n \,, 
  \\
  \displaystyle
  \frac{2n-|s|}{n} & \mbox{if} \quad n \le |s| \le 2n \,, 
  \\
  0 & \mbox{if} \quad |s| > 2n \,.
\end{cases} 
\end{equation}
Notice that~$\varphi_n$ converges pointwise to~$1$ as $n \to \infty$.

Integrating by parts and 
using the differential equation that~$\psi_1$ satisfies,
one verifies that
\begin{equation}\label{yields.bis}
  h[\Psi_n] - \lambda_1 \, \|\Psi_n\|_\Hilbert^2
  = \left\|\frac{\varphi_n' \psi_1}{f} \right\|_\Hilbert^2
  + \big(\varphi_n \psi_1,W\varphi_n \psi_1'\big)_\Hilbert
  \,,
\end{equation}
where
\begin{equation}
  W(s,t) := \frac{1}{t} - \frac{\partial_t f(s,t)}{f(s,t)}
  = \frac{1}{t \, \big[1+\theta'(s)^2\, t^2\big]}
  \,.
\end{equation}
Note that $W\varphi_n \psi_1'$ belongs to $\Hilbert$.
Indeed, the fact that $t^{-1} \psi_1'(t)$
remains bounded as $t \to 0$ follows from the Neumann
condition that~$\psi_1$ satisfies at zero, \cf~\eqref{Neumann}.

By the variational definition of $\inf\sigma(H)$,
it is enough to show that the right-hand side of~\eqref{yields.bis}
is negative for some~$n$.	
Using the pointwise inequality $f \ge 1$, 
we obtain
\begin{equation}
  \left\|\frac{\varphi_n' \psi_1}{f} \right\|_\Hilbert^2
  \le \|\varphi_n'\|_{\sii(\Real)}^2 \, \|\psi_1\|_{\sii((a_1,a_2))}^2
  = n^{-1} \, \|\psi_1\|_{\sii((a_1,a_2))}^2
  \,,
\end{equation}
so this term goes to zero as $n\to\infty$.
To show that the second term on the right-hand side of~\eqref{yields.bis}
is negative, we use that~$\psi_1$ is decreasing on $(0,a_2)$
and increasing on $(a_1,0)$.
(In fact, because of the decoupling, 
$\psi_1$~is necessarily identically zero in one of these intervals.)
At the same time, $t \mapsto W(s,t)$ is positive on $(0,a_2)$
and negative on $(a_1,0)$.
Consequently, 
\begin{equation}
  c_n := \big(\varphi_n \psi_1,W\varphi_n \psi_1'\big)_\Hilbert
  < 0
  \,.
\end{equation}
Moreover, $n \mapsto c_{n}$ is decreasing
due to the definition of~$\varphi_n$.
Hence, the right-hand side of~\eqref{yields.bis} is negative
for all sufficiently large~$n$ and the desired claim follows. 
\end{proof}

While Theorem~\ref{Thm.disc} is valid for any twisting angle~$\theta$,
it is of particular interest for diverging twisting~\eqref{diverge},
when~$\lambda_1$ is the threshold of the essential spectrum of~$H$ 
due to Theorem~\ref{Thm.ess}.
In this case Theorem~\ref{Thm.disc} implies 
Theorem~\ref{Thm.disc.intro} as a corollary.

%
%
\providecommand{\bysame}{\leavevmode\hbox to3em{\hrulefill}\thinspace}
\providecommand{\MR}{\relax\ifhmode\unskip\space\fi MR }
\providecommand{\MRhref}[2]{%
  \href{http://www.ams.org/mathscinet-getitem?mr=#1}{#2}
}
\providecommand{\href}[2]{#2}

\end{document}